\documentclass{article}

\usepackage[centertags]{amsmath}
\usepackage{amsfonts}
\usepackage{amssymb}
\usepackage{amsthm}
\usepackage{newlfont}
\usepackage{mathtools}
\usepackage[top=1in, bottom=1in, left=1in, right=1in]{geometry}
\usepackage{tikz}

\theoremstyle{plain}
\newtheorem{theorem}{Theorem}[section]
\newtheorem{proposition}[theorem]{Proposition}
\newtheorem{corollary}[theorem]{Corollary}
\newtheorem{lemma}[theorem]{Lemma}

\theoremstyle{definition}
\newtheorem{definition}[theorem]{Definition}

\newcommand{\R}{\mathbb{R}}
\newcommand{\N}{\mathbb{N}}
\newcommand{\Z}{\mathbb{Z}}

\newcommand{\Q}{\mathbb{Q}}
\newcommand{\F}{\mathbb{F}}
\newcommand{\Cy}{\mathcal{C}}
\newcommand{\SPAN}{\operatorname{span}}
\newcommand{\de}{\partial}

\newcommand{\Hidden}[1]{}

\newcommand{\Bu}{\mathcal{B}}

\newcommand{\Hm}[1]{\leavevmode{\marginpar{\tiny%
$\hbox to 0mm{\hspace*{-0.5mm}$\leftarrow$\hss}%
\vcenter{\vrule depth 0.1mm height 0.1mm width \the\marginparwidth}%
\hbox to 0mm{\hss$\rightarrow$\hspace*{-0.5mm}}$\\\relax\raggedright
#1}}}

\begin{document}

\title{Relationships between cycle spaces, gain graphs, graph coverings, fundamental groups, path homology, and graph curvature}
\author{Mark Kempton\footnote{Center of Mathematical Sciences and Applications, Harvard University, Cambridge MA, USA. mkempton@cmsa.fas.harvard.edu. Supported by Air Force Office of Scientific Research Grant FA9550-13-1-0097: Geometry and Topology of Complex Networks.}~~~~~Florentin M\"unch\footnote{Universit\"at Potsdam, Potsdam, Germany. chmeunch@uni-potsdam.de.  Supported by the German National Merit Foundation}~~~~~Shing-Tung Yau\footnote{Harvard University, Cambridge MA, USA. yau@math.harvard.edu.}}
\date{}
\maketitle

\begin{abstract}
We prove a homology vanishing theorem for graphs with positive Bakry-\'Emery curvature, analogous to a classic result of Bochner on manifolds \cite{Bochner}.  Specifically, we prove that if a graph has positive curvature at every vertex, then its first homology group is trivial, where the notion of homology that we use for graphs is the path homology developed by Grigor'yan, Lin, Muranov, and Yau \cite{Grigoryan2}. 
We moreover prove that the fundamental group is finite for graphs with positive Bakry-\'Emery curvature, analogous to a classic result of Myers on manifolds \cite{Myers1941}. 
The proofs draw on several separate areas of graph theory.  We study graph coverings, gain graphs, and cycle spaces of graphs, in addition to the Bakry-\'Emery curvature and the path homology.  The main results follow as a consequence of several different relationships developed among these different areas.  Specifically, we show that a graph with positive curvature can have no non-trivial infinite cover preserving 3-cycles and 4-cycles, and give a combinatorial interpretation of the first path homology in terms of the cycle space of a graph.  We relate cycle spaces of graphs to gain graphs with abelian gain group, and relate these to coverings of graphs.  Along the way, we prove other new facts about gain graphs, coverings, and cycles spaces that are of related interest.  Furthermore, we relate gain graphs to graph homotopy and the fundamental group developed by Grigor'yan, Lin, Muranov, and Yau \cite{Grigoryan_homotopy}, and obtain an alternative proof to their result that the abelianization of the fundamental group is isomorphic to the first path homology over the integers.
\end{abstract}


\section{Introduction}

A significant theme in much of graph theory in recent years has been the application of tools and ideas from continuous geometry to discrete settings, most specifically to graphs.  For instance, a classical example resulting from this way of thinking is the well-know Cheeger inequality (see for instance \cite{Chung97}), which proves an isoperimtric inequality for graphs that was originally formulated for Riemannian manifolds.  There has been growing interest in recent years both in the approximation of continuous spaces by discrete ones, and in the understanding of graphs via their geometric properties. 

One of the principal developments in this area concerns curvature for graphs.  Numerous notions of curvature on graphs have been put forward \cite{Forman,ollivier2009ricci}.  An important and very general notion of curvature for graphs has been defined via various formulas due to Bakry and \'Emery, which is called the Bakry-\'Emery curvature of a graph (see \cite{BakryEmery85,schmuckenschlager1998curvature,
LinYau10}).

In addition, there are various notions of homology and cohomology for graphs.  Recent work has introduced one such theory called the path homology \cite{Grigoryan1}.  Path homology has been shown to be a non-trivial homology theory which is invariant under a notion of homotopy for graphs \cite{Grigoryan_homotopy}.  Using this homotopy theory, in \cite{Grigoryan_homotopy} the fundamental group for a graph is defined, and it is shown that the first path homology is isomorphic to the abelianization of this fundamental group.  Furthermore, it satisfies nice functorial properties, namely the K\"unneth formula holds for graph products \cite{Grigoryan3}.  For these reasons, it seems that the path homology is a more appropriate notion of homology for graphs than others that have been proposed.  See \cite{Grigoryan3} for a discussion of various homology theories for graphs and the advantages of the path homology.  

In this paper, we prove an important connection between these two notions, namely we prove a homology vanishing theorem for graphs with positive Bakry-\'Emery curvature.  Homology vanishing theorems are ubiquitous in continuous geometry, and give important structural information about manifolds.  Our vanishing theorem is analogous to a fundamental result of Bochner on manifolds \cite{Bochner} .


\begin{theorem}\label{hom_van}
If a finite graph $G$ has positive Bakry-\'Emery curvature at every vertex, then its first path homology group is trivial.
\end{theorem}

It turns out that Bakry-\'Emery curvature on graphs is also compatible with the notion of homotopy and fundamental group following \cite{Grigoryan_homotopy}.
Our homotopy theorem is analogous to a fundamental result of Myers on manifolds \cite{Myers1941} .

\begin{theorem}\label{thm:fundamentalCurv}
If a finite graph $G$ has positive Bakry-\'Emery curvature at every vertex, then its fundamental group $\pi_1(G)$ is finite.
\end{theorem}

Showing that these classical theorems from geometry hold for graphs reinforces the idea that the path homology is the correct notion of homology for graph theory.  

While Theorems~\ref{hom_van} and \ref{thm:fundamentalCurv} can be considered the main results of this paper, our proofs are executed by developing relationships between several different areas of graph theory, several of which are interesting on their own.  Particularly, we develop relationships between gain graphs which are graphs with edges labeled with elements of a group, and the cycle space of a graph, which is the space of linear combinations of incidence vectors of cycles of a graph.  The cycle space of a graph always has a basis of size equal to the cyclomatic number of the graph, but there are various different classes of cycle bases depending on certain properties, which may not be satisfied for all bases.  See \cite{liebchen2007classes} for a discussion of different kinds of cycle basis, and for a determinantal characterization of these various bases.  One of our contributions is to give a new kind of cycle basis based off of a gain graph, which we call a $\Gamma$-circuit generator, where $\Gamma$ is the group associated with a gain graph.  We prove relationships between $\Gamma$-circuit generators and other classes of cycle bases, and give a determinantal  characterization of $\Gamma$-circuit generators.  We also relate gain graphs to covering of graph that preserve cycles.  Furthermore, we relate the path homology of a graph to the cycle space by showing that the first path homology is isomorphic to the cycle space modulo the space generated by simple cycles of length 3 and 4. This gives us a combinatorial interpretation of the first path homology group.  It is an open question to find a similar nice interpretation of the path homology groups beyond the first. 

In addition, we investigate the fundamental group as defined in \cite{Grigoryan_homotopy}, and give an interpretation of this group relating to gain graphs. Then, via results of \cite{DeVosFunkPivotto14}, we are able to describe this fundamental group as the fundamental group of the topological space obtained by attaching a 2-cell to each cycle of length 3 or 4 in $G$.
This allows us to connect graph coverings with the fundamental group.
  Our results also give an alternative proof of the result in \cite{Grigoryan_homotopy} that the first path homology over $\Z$ is isomorphic to the abelianization of this fundamental group.

\subsection{Organization and main results}

The remainder of this paper will be organized as follows.  In Section \ref{sec:prelim}, we will give the technical preliminaries, including definitions and known results concerning gain graphs, cycle bases, covers of graphs, graph homology, and graph curvature.  In addition to relevant known results, we present some new lemmas that will be useful later.  

In Section \ref{sec:CurvCover}, we connect the Bakry-\'Emery curvature to coverings of graph.  Using a known diameter bound involving the curvature \cite{liu2016bakry,FathiYan2018}, we prove that a graph with positive curvature has no infinite covering that preserves 3- and 4-cycles. 

Section \ref{sec:CoverCircuit} explores the relationship between gain graphs and covers of graphs. In particular, we use a construction due to \cite{GrossTucker77} to produce coverings of graphs corresponding to gain functions, and show that those coverings preserve precisely the cycles that are balanced under the gain function. Furthermore, every graph covering arises in this way.

Section \ref{sec:BasisGain} investigates the relationships between cycle spaces and gain graphs.  
In particular we prove that a collection of circuits (with size the cyclomatic number of the graph) is an $\F$-cycle basis for the cycle space if and only if is is an $\Gamma$-circuit generator, where $\Gamma$ is the additive group of the field $\F$.  Furthermore, given such a collection of circuits $\Bu$, we associate a determinant that we call $\det \Bu$, and show that for an Abelian group $\Gamma$, that $\Bu$ is a $\Gamma$-circuit generator if and only if $g^{\det\Bu}\neq e_\Gamma$ for all non-identity elements $g\in \Gamma$.  This characterizes $\Gamma$-circuit generators for Abelian groups via a determinant, similar to the situation for other classes of cycle basis.

 In Section \ref{sec:BasisHomology}, we prove a relationship between the path first homology group of a graph and its cycle space. Namely, we show that the first homology group $H_1(G,\F)$, for a field $\F$, is isomorphic to the $\F$-cycle space modulo the space generated by all 3- and 4-cycles.  This gives an interpretation of what the first path homology group is ``measuring" in terms of a well-studied combinatorial concept--the cycle space.  
With these results, we prove Theorem \ref{hom_van}.  We further discuss the converse and possible weakening of the hypotheses of Theorem \ref{hom_van}.  We end this section with a discussion of a different homology theory for graphs--the clique homology.  We give a similar combinatorial interpretation of the first clique homology group, and point out that Theorem \ref{hom_van} does not hold if we use the clique homology instead of the path homology.  This strengthens the notion that the path homology has many advantages over other graph homology theories.  

Finally, Section \ref{sec:homotopy} investigates the fundamental group of a graph $\pi_1(G)$ under the notion of homotopy from \cite{Grigoryan_homotopy}.  This notion of homotopy treats 3- and 4-cycles as contractible subgraphs.  We show that this fundamental group is isomorphic to a quotient of a canonical gain group balanced on the set of 3- and 4-cycles.  Indeed, we define a generalization of the fundamental group, treating any arbitrary collection of cycles as contractible, and show that this is similarly isomorphic to a quotient of a canonical gain graph with the same set of cycles.  
As in classical topology, we connect the fundamental group with the universal covering allowing us to prove Theorem~\ref{thm:fundamentalCurv}.
We also show how our results give an alternate proof Theorem 4.23 of \cite{Grigoryan_homotopy}, which says that $H_1(G,\Z)$ is the abelianization of $\pi_1(G)$.

\section{Preliminaries}\label{sec:prelim}

\subsection{Gain graphs and circuit generators}

Let $G=(V,E)$ an undirected graph. We will denote by $\vec E$ the set that contains two directed arcs, one in each direction, for each edge in $E$. Let $\Gamma$ be a group.
A \emph{gain graph} is a triple $(G,\phi,\Gamma)$ where
$\phi: \vec{E} \to  \Gamma$ is a map satisfying $\phi(xy) = \phi(yx)^{-1}$ for all edges $(xy) \in \vec{E}$.  The map $\phi$ is called the \emph{gain function} of the gain graph.
Denote by $\Phi(G,\Gamma)$ the set of all gain functions from the graph $G$ to the group $\Gamma$.

Gain graphs have also been referred to as \emph{voltage graphs}, and are special cases of \emph{biased graphs} (see \cite{Zaslavsky}).
When $\Gamma$ is a group of invertible linear transformations, they are also called \emph{connection graphs} \cite{ChungZhaoKempton}, and the map $\phi$ can be considered as a connection corresponding to a vector bundle on the graph \cite{Kenyon}.

For the most part, we will take terminology about gain graphs and biased graphs from \cite{Zaslavsky}. A circuit or simple cycle $C$ is a simple closed walk $(x_1,\ldots,x_n)$.  
Throughout the paper, we will refer to circuits of length 3 as ``triangles," and circuits of length 4 as ``squares."
We write $\mathcal S(G)$ for the set of all circuits in $G$. A \emph{theta graph} is the union of three internally disjoint simple paths that have the same two distinct endpoint vertices.
A \emph{biased graph} is a pair $(G,\mathcal B)$ where $\mathcal B\subset \mathcal S(G)$ is a set of distinguished circuits, called \emph{balanced circuits}, that form a \emph{linear subclass} of circuits, that is, $\Bu$ has the property that if any two circuits of a theta graph are in $\mathcal B$, then so is the third.  
We say $\Bu \subset \mathcal S(G)$ is a \emph{cyclomatic circuit set} if the cardinality of $\Bu$ equals the cyclomatic number of $G$, $|E|-|V|+1$.

Gain graphs are biased graphs in a natural way.
Define define the \emph{order} of a circuit $C=(x_1,\ldots,x_n)$ under a gain function $\phi$ via
\[
o_\phi(C) := \inf \left\{r>0:  \left[ \phi(x_1x_2) \ldots \phi(x_{n-1}x_n) \phi(x_n x_1) \right]^r = e_\Gamma \right\}.
\]
We say the gain function $\phi$ is \emph{balanced} on a circuit $C=(x_1,\ldots,x_n)$ if $o_\phi(C)=1$.
Denote by $\mathcal B(\phi)$ the set of balanced circuits.  Then $(G,\mathcal B(\phi))$ defines a biased graph.

\begin{definition}[$\Gamma$-circuit generator]
We say a set $\Bu$ of circuits is a \emph{$\Gamma$-circuit generator} if for all $\phi \in \Phi(G,\Gamma)$, $\phi$ balanced on $\Bu$ implies that $\phi$ is balanced on the entire graph $G$.  
\end{definition}

\begin{definition}[Canonical gain graph]
Let $G=(V,E)$ be a graph and let $\mathcal B$ be a set of circuits.
We define the group $\Gamma(G,\mathcal B)$ via the presentation
\[
\Gamma(G,\mathcal B) = \langle \vec E | \mathcal B \rangle,
\]
i.e., $\Gamma(G,\mathcal B)$ is generated by the oriented edges and each circuit $C=(x_1,\ldots,x_n) \in \mathcal B$ gives a relation
$(x_1x_2) \ldots (x_{n-1}x_n)(x_n x_1) = e_\Gamma$ where we identify $(xy)=(yx)^{-1} \in \vec E$.
There is a natural gain function $\phi_{\mathcal B}$  given by the natural mapping of $\vec E$ into $\Gamma(G,\mathcal B)$, and the corresponding gain graph $(G,\phi_{\mathcal B},\Gamma(G,\mathcal B))$ is called the \emph{canonical gain graph} associated with the biased graph $(G,\mathcal B)$.  
\end{definition}
In \cite{Zaslavsky2}, it is shown that the canonical gain graph satisfies a universal property with respect to gain graphs balanced on $\mathcal B$.

\begin{proposition}[Theorem 2.1 of \cite{Zaslavsky2}]\label{prop:universal}
Given any gain graph $(G,\psi,\Gamma)$ such that $\psi$ is balanced on $\Bu$, then there exists a homomorphism $h:\Gamma(G,\Bu)\rightarrow\Gamma$ such that $\psi = h \circ \phi_\Bu$ as defined above.
\end{proposition}

\Hidden{ 
Given a subset $S\subset E$, the \emph{balanced closure} of $S$, $\bcl(S)$ is defined by
\[
\bcl(S) = S\cup\left\{e\in E: \text{ there is a balanced circuit } C \text{ such that } e\in C\subseteq S\cup\{e\}\right\}.
\]
We will also consider the balanced closure of a set of circuits to be the balanced closure of the set of all edges belonging to those circuits.  
} 

\begin{definition}[Combinatorial circuit generator]
We say, $\Bu \subset \mathcal S$ is a \emph{combinatorial circuit generator} if the linear subclass generated by $\Bu$ is the entire set of circuits $\mathcal S(G)$.
\end{definition}

For convenience in talking about combinatorial circuit generators and linear subclasses of circuits, we define the following operation on circuits.  For circuits $C_1,C_2$ belonging to the same theta graph, we define $C_1\oplus C_2$ to be the third cycle in the theta graph.  Then we may describe a linear subclass as a set of circuits that is closed under the operation $\oplus$.

\begin{proposition}\label{prop:products}
Let $\Gamma_1,\Gamma_2$ be groups, let $G$ be a finite graph and let $\mathcal B \subset \mathcal S(G)$ be a set of cycles.
T.f.a.e.
\begin{enumerate}
\item
$\mathcal B$ is a $\Gamma_1 \times \Gamma_2$ circuit generator.
\item $\mathcal B$ is a $\Gamma_1$ circuit generator and a $\Gamma_2$ circuit generator.
\end{enumerate}
Moreover if $\Gamma_1 \subset \Gamma_2$ is a subgroup and if $\mathcal B$ is a $\Gamma_2$ circuit generator, then $\mathcal B$ is also a $\Gamma_1$ circuit generator.
\begin{proof}
The fact that the circuit generator property is inherited to subgroups directly follows from the definition. This also proves 1 $\Rightarrow$ 2.
We now prove 2 $\Rightarrow$ 1 indirectly.
Suppose $\mathcal B$ is not a $\Gamma_1 \times \Gamma_2$ circuit generator. Then, there exists $C=(x_1,\ldots,x_n) \in \mathcal S(G)$ and $\phi \in \Phi(G,\Gamma_1 \times \Gamma_2)$ s.t. $\phi$ is balanced on $\mathcal B$ but not on $C$, i.e.,
$\phi(x_1x_2)\ldots \phi(x_{n-1}x_n)\phi(x_nx_1) = (g_1,g_2) \in \Gamma_1 \times \Gamma_2$ with $g_k \neq e_{\Gamma_k}$ for some $k \in \{1,2\}$.
Define $\phi_i \in \Phi(G,\Gamma_i)$ s.t. $\phi(e) = (\phi_1(e),\phi_2(e))$ Now, $\phi_i$ is balanced on $\mathcal B$ for $i=1,2$ and $\phi_k$ is not balanced on $C$. Therefore, $\mathcal B$ is not a $\Gamma_k$-circuit generator. This is a contradiction and finishes the proof.
\end{proof}

\end{proposition}

We prove some facts about circuit generators that will be useful to us later on.

\begin{proposition}\label{prop:finitelyGenerated}
Let $\Gamma_0$ be a group, let $G$ be a finite graph and let $\mathcal B \subset \mathcal S(G)$ be a set of cycles. T.f.a.e.:
\begin{enumerate}
\item
$\mathcal B$ is a $\Gamma_0$-circuit generator.
\item
$\mathcal B$ is a $\Gamma$-circuit generator for all finitely generated subgroups of $\Gamma_0$.
\end{enumerate}
\end{proposition}

\begin{proof}
1 $\Rightarrow$ 2 follows from Proposition~\ref{prop:products}.

2 $\Rightarrow$ 1: Suppose $\mathcal B$ is not a $\Gamma_0$-circuit generator. Then there exists $\phi \in \Phi(G,\Gamma_0)$ that is balanced on $\mathcal B$ but not on some $C \in \mathcal S(G)$. Let $\Gamma$ be the group generated by all $\phi(e)$ for all $e \in \vec E$. Then, $\Gamma$ is a finitely generated subgroup of $\Gamma_0$ and $\phi \in \Phi(G,\Gamma)$ is balanced on $\mathcal B$ but not on $C$. Therefore, $\mathcal B$ is not a $\Gamma$-circuit generator. This finishes the proof.
\end{proof}

\subsection{Cycle bases and fields}

If $\F$ is a field with additive group $\Gamma$, then we write $\Phi(G,\F):=\Phi(G,\Gamma)$ being a $\F$-vector space.

The $\F$-cycle space is defined by
\[
\Cy(G,\F):= \left\{ \phi \in \Phi(G,\F): \sum_{y:(xy) \in \vec E} \phi(xy)=0 \mbox{ for all } x \in V \right\}.
\]

Every circuit can be identified with a cycle $\phi \in \Cy(G,\F)$ s.t. $\phi(\vec E) \subset \{0,\pm1\}$ and such that $\phi^{-1}(\{1\})$ is a simple connected cycle. Therefore, we have $\mathcal S \hookrightarrow  \Cy(G,\F)$ allowing us to add circuits and to multiply them with scalars.
For a cycle $C \in \mathcal C(G,\F)$ and a gain function $\phi \in \Phi(G,\F)$ define
\[
\phi(C) := \sum_{e \in \vec E} \phi(e)C(e).
\]

\begin{definition}[$\F$-cycle basis]
A subset $\mathcal B \subset \Cy(G,\F)$ is called an \emph{$\F$-cycle basis} of $G$ if $\mathcal B$ is a vector space basis of  $\Cy(G,\F)$.
In the literature, a $\Q$-cycle basis is also called a \emph{directed} cycle basis and an $\F_2$-cycle basis is also called an \emph{undirected} cycle basis.
\end{definition}

According to \cite{liebchen2007classes}, we define the determinant of a set of cycles.
Let $r$ be the cyclomatic number of $G$ and let $\mathcal B \subset \mathcal S(G)$ be of size $r$.
Corresponding to \cite[Definition~22]{liebchen2007classes}, consider the matrix $M(\mathcal B,\F)$ over the field $\F$ with the incidence vectors of $\mathcal B$ as columns.
Let $M(\mathcal B,\F,T)$be the $r \times r$ submatrix that arises when deleting the arcs of the spanning $T$ tree of $G$.
Remark that $M(\mathcal B,\F)$ consists only of the entries $0$ and $\pm 1$.
Now write
$$
\det \mathcal B := |\det M(\mathcal B,\Q,T)|.
$$
It is shown in \cite{liebchen2007classes} that $\det \mathcal B$ does not depend on the choice of the spanning tree $T$.
The following theorem is a simple generalization of the characterization of directed and undirected cycle basis via determinants (see \cite{liebchen2007classes}).
\begin{theorem}\label{thmCycleDet}
A set $\mathcal B \subset \mathcal C(G,\F)$ is an $\F$-basis if and only if $\det \mathcal B \not\equiv 0 \mod \chi(\F)$ where $\chi(\F)$ denotes the characteristic of the field $\F$.
\end{theorem}
\begin{proof}
It is easy to see that $\mathcal B$ is a $\F$-vector space basis of $\mathcal C(G,\F)$ if and only if $M(\mathcal B,\F,T)$ is invertible as a matrix over  $\F$ for a given spanning tree $T$ of $G$. This holds true if and only if $\det M(\mathcal B,\F,T) \neq 0$ which is equivalent to
\[
\det M(\mathcal B,\Q,T) \not\equiv 0 \mod \chi(\F).
\]
This directly implies the theorem.
\end{proof}

Besides the $\F$ cycle basis, there are also stricter notions of cycle basis. The following definitions can be found in \cite{liebchen2007classes}.

\begin{definition}[Integral cycle basis] A set $\mathcal B = \{C_1, . . . , C_r\} \subset \mathcal C(G,\Q)$ of oriented circuits of a graph $G$ is an \emph{integral cycle basis} of $G$ if every oriented circuit C of $G$ can be written as an integer linear combination of circuits in $\mathcal B$, i.e., there exist $\lambda_i \in \Z$ s.t.
\[
C= \lambda_i C_i + \ldots + \lambda_r C_r.
\] 
\end{definition}

\begin{proposition}[\cite{liebchen2003finding}]
Let  $\mathcal B = \{C_1, . . . , C_r\} \subset \mathcal C(G,\Q)$ be a set of oriented circuits of a graph $G$. T.f.a.e.:
\begin{enumerate}
\item
$\mathcal B$ is an integral basis.
\item $\det \mathcal B$ = 1.
\end{enumerate}
\end{proposition}

\begin{definition}[Totally unimodular cycle basis] A $\Q$-cycle basis $\mathcal B = \{C_1, . . . , C_r\} \subset \mathcal C(G,\Q)$ of a graph $G$ is a \emph{totally unimodular cycle basis} of $G$ if its cycle matrix $M(\mathcal B, \Q)$ is totally unimodular, i.e., each sub determinant is either $0$ or $\pm 1$.
\end{definition}

\begin{definition}[Weakly fundamental cycle basis]
A set $\mathcal B = \{C_1,\ldots,C_r\}$ of circuits  of a graph $G$ is a \emph{weakly fundamental cycle basis} of $G$ if there exists some permutation $\sigma$ such that for all $i=2,\ldots,r$,
\[
C_{\sigma(i)} \setminus (C_{\sigma(1)} \cup \ldots \cup C_{\sigma(i-1)}) \neq \emptyset.
\]
\end{definition}

\begin{proposition}[\cite{liebchen2007classes}]
Let  $\mathcal B = \{C_1, . . . , C_r\} \subset \mathcal C(G,\Q)$ be a set of oriented circuits of a graph $G$. T.f.a.e.:
\begin{enumerate}
\item
$\mathcal B$ is a weakly fundamental cycle basis.
\item There exists a spanning tree $T$ and a permutation of columns and rows such that $M(\mathcal B,\Q,T)$ is lower triangular.
\end{enumerate}
\end{proposition}

\begin{definition}[Strictly fundamental cycle basis]
A set $\mathcal B$ of circuits of a graph $G$ is a \emph{strictly fundamental cycle basis} of $G$, if there exists some spanning tree $T \subseteq E$ such that $\mathcal B = \{C_e : e \in E \setminus T \}$, where $C_e$ denotes the unique circuit in $T \cup \{e\}$.
\end{definition}

\begin{proposition}[\cite{liebchen2007classes}]
Let  $\mathcal B = \{C_1, . . . , C_r\} \subset \mathcal C(G,\Q)$ be a set of oriented circuits of a graph $G$. T.f.a.e.:
\begin{enumerate}
\item
$\mathcal B$ is a strictly fundamental cycle basis.
\item There exists a spanning tree $T$ and a permutation of columns and rows such that $M(\mathcal B,\Q,T)$ is diagonal.
\end{enumerate}
\end{proposition}

\Hidden{
\subsection{Combinatorial circuit generators}
We say, two circuits $C_1,C_2$ share a path if there exists paths $p_1,p_2,p_3$ s.t. $C_i$ can be written as $C_i=(p_i,p_3)$ for $i=1,2$.
For two circuits $C_1,C_2$, we define $\oplus: \mathcal S \times \mathcal S \to \mathcal S$ via
\[
C_1 \oplus C_2 := \begin{cases}
C_1 - C_2 &: C_1 \mbox{ and } C_2 \mbox{ share a path}\\
0 &:\mbox{ otherwise}.
\end{cases}
\]

We say a set $\mathcal B \subset \mathcal S$ is closed under $\oplus$ if for all $C_1,C_2 \in \mathcal B$ also $C_1\oplus C_2 \in \mathcal B$.
Remark that the intersection of sets closed under $\oplus$ are also closed. Therefore it is reasonable to define the closure of a set $\mathcal B$ as the intersection of all supersets of $\mathcal B$ being closed under $\oplus$.

\begin{definition}[Combinatorial circuit generator]
We say, $\mathcal B \subset \mathcal S$ is a \emph{combinatorial circuit generator} if the closure of  $\mathcal B$ under $\oplus$ is the circuit space $\mathcal S(G)$.
\end{definition}
}

\subsection{Circuit preserving coverings}
Let $G=(V,E)$ be a connected graph.  If $\widetilde G = (\widetilde V, \widetilde E)$ is a graph, and $\Psi:\widetilde G\rightarrow G$ is a surjective graph homomorphism such that $\Psi$ is locally bijective (i.e., $\Psi$ is bijective when restricted to the neighborhood of a single vertex), then the pair $(\widetilde G,\Psi)$ is called a \emph{covering} of $G$. 

Since $\Psi$ is locally bijective, than it can be seen that $|\Psi^{-1}(x)|$ is constant for all vertices $x\in V$.  If this constant value is $m$, we say that $(\widetilde G,\Psi)$ is a covering with \emph{$m$ sheets}, or is an \emph{$m$-sheeted} covering.  Here, $m$ can be infinite. 

We call a covering $(\widetilde G,\Psi)$ \emph{trivial} if $\Psi$ restricted to any connected component of $\widetilde G$ is a graph isomorphism.  We say it is \emph{non-trivial} otherwise, i.e., if there is at least one connected component of $\widetilde G$ on which $\Psi$ is not one-to-one.  

Let $\mathcal B \subset \mathcal S(G)$ be a set of circuits. 
We say a covering $(\widetilde G, \Psi)$ is a $\mathcal B$ \emph{preserving covering} of $G$ if for all circuits $C=(x_1,\ldots,x_n) \in \mathcal B$ and all $\widetilde x_1 \in \widetilde V$ with $\Psi(\widetilde x_1) = x_1$, there exist a circuit $\widetilde C = (\widetilde x_1,\ldots, \widetilde x_n) \in \widetilde V$ s.t. $\Psi(\widetilde x_k) = x_k$ for all $k$. Note in particular that every circuit in the pre-image of $C$ has length equal to the length of $C$.

\subsection{Path homology of graphs}
In this section, we will give definitions for a homology theory on graphs that has been developed in recent years, called \emph{path homology}.  See \cite{Grigoryan1,Grigoryan2,Grigoryan3}.  This homology theory is most naturally described for directed graph, and the homology of an undirected graph is obtained by orienting each edge of an undirected graph in both possible directions.  

For a directed graph $G=(V,E)$ (without self-loops) we start by defining an \emph{elementary $m$-path} on $V$ to be a sequence $i_0,...,i_m$ of $m+1$ vertices of $V$.  For a field $\F$ we define the $\F$-linear space $\Lambda_m$ to consist of all formal linear combinations  of elementary $m$-paths with coefficients from $\F$.  We identify an elementary $m$-path as an element of $\Lambda_m$ denoted by $e_{i_0...i_m}$, and $\{e_{i_0...i_m}:i_0,...,i_m\in V\}$ is a basis for $\Lambda_m$.  Elements of $\Lambda_m$ are call \emph{$m$-paths}, and a typical $m$-path $p$ can be written as
\[
p = \sum_{i_0,...,i_m\in V}a_{i_0...i_m}e_{i_0...i_m},~~~a_{i_0...i_m}\in\F.
\]
Note that $\Lambda_0$ is the set of all formal linear combinations of vertices in $V$. 

We define the \emph{boundary operator} $\de:\Lambda_m\rightarrow\Lambda_{m-1}$ to be the $\F$-linear map that acts of elementary $m$-paths by
\[
\de e_{i_0...i_m} = \sum_{k=0}^m (-1)^ke_{i_0...\hat i_k...i_m},
\]
where $\hat i_k$ denotes the omission of index $i_k$.  

For convenience, we define $\Lambda_{-1} = 0$ and $\de:\Lambda_0\rightarrow\Lambda_{-1}$ to be the zero map.

It can be checked that $\de^2 =0$, so that the $\Lambda_m$ give a chain complex (see \cite{Grigoryan3}).  When it is important to make the distinction, we will use $\de_m$ to denote the boundary map on $\Lambda_m$, $\de_m:\Lambda_m\rightarrow\Lambda_{m-1}$.

An elementary $m$-path $i_0...i_m$ is called \emph{regular} if $i_k\neq i_{k+1}$ for all $k$, and is called \emph{irregular} otherwise.  Let $I_m$ be the subspace of $\Lambda_m$ spanned by all irregular $m$-paths, and define 
\[
\mathcal R_m = \Lambda_m/I_m.
\]
The space $\mathcal R_m$ is isomorphic to the span of all regular $m$-paths, and the boundary map $\de$ is naturally defined on $\mathcal R_m$, treating any irregular path resulting from applying $\de$ as 0.

In the graph $G=(V,E)$, call an elementary $m$-path $i_0...i_m$ \emph{allowed} if $i_ki_{k+1} \in E$ for all $k$.  Define $\mathcal A_m$ to be the subspace of $\mathcal R_m$ given by
\[
\mathcal A_m = \SPAN\{e_{i_0...i_m}:i_0...i_m \text{ is allowed}\}.
\]
The boundary map $\de$ on $\mathcal A_m$ is simply the restriction of the boundary map on $\mathcal R_n$, however, it can be the case that the boundary of an allowed $m$-path is not an allowed $(m-1)$-path.  So we make one further restriction, and call an elementary $m$-path $p$ \emph{$\de$-invariant} if $\de p$ is allowed.  We define 
\[
\Omega_m = \{p\in \mathcal A_m : \de p \in \mathcal A_{m-1}\}.
\]
Then it can be seen that $\de\Omega_m\subseteq \Omega_{m-1}$.  The $\Omega_m$ with the boundary map $\de$ give us our chain complex of $\de$-invariant allowed paths from which we will define our homology:
\[
\cdots\Omega_{m}\overset{\de}{\rightarrow}\Omega_{m-1}\overset{\de}{\rightarrow}\cdots\rightarrow\Omega_1\rightarrow\Omega_0\rightarrow0.
\]
Observe that $\Omega_0$ is the space of all formal linear combinations of vertices of $G$, and $\Omega_1$ is space of all formal linear combinations of edges of $G$.  We can now define the homology groups of this chain complex.

\begin{definition}[path homology]
The \emph{path homology groups} of the graph $G$ over the field $\F$ are
\[
H_n(G,\F) = Ker~\de|_{\Omega_n}/Im~\de|_{\Omega_{n+1}}.
\]
\end{definition}

A standard fact is that $dim~H_0(G,\F)$ counts the number of connected components of $G$ (\cite{Grigoryan1}).

As a standard example of some of the interesting behavior of this homology, consider the directed 4-cycle pictured below.
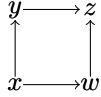
\begin{center}
\begin{tikzpicture}
\draw[shorten >=4pt,shorten <=3pt, ->](0,0)node{$x$}--(0,1)node{$y$};
\draw[shorten >=4pt,shorten <=3pt, ->](0,1)node{$y$}--(1,1)node{$z$};
\draw[shorten >=4pt,shorten <=3pt, ->](0,0)node{$x$}--(1,0)node{$w$}; 
\draw[shorten >=4pt,shorten <=3pt, ->](1,0)node{$w$}--(1,1)node{$z$};
\end{tikzpicture}
\end{center}
Note that the 2-paths $e_{xyz}$ and $e_{xwz}$ are allowed by not $\de$-invariant, since the edge $xz$ is missing from the graph.  However, if we consider the linear combination $e_{xyz}-e_{xwz}$, then note
\[
\de(e_{xyz} - e_{xwz}) = e_{yz}-e_{xz}+e_{xy} - (e_{wz}-e_{xz}+e_{xw}) = e_{yz}+e_{xy}-e_{wz}+e_{xw} \in \Omega_1.
\]
Thus $e_{xyz}-e_{xwz}\in\Omega_2$, and it turns out that $\Omega_2 = span\{e_{xyz}-e_{xwz}\}$.  It can be seen that $Ker~\de_1$ is spanned by $e_{xy}+e_{yz}-e_{zw}-e_{xz}$, which is precisely $Im~\de_2$, so $H_1(G,\F)=0$.  For circuits of length more than 4, then $dim~H_1(G,\F) = 1$.  See \cite{Grigoryan3} for details.  

\subsection{Graph homotopy and fundamental group}

There is a separate notion on homotopy for graphs \cite{Grigoryan_homotopy} under which the path homology is invariant.  Via this homotopy, one can define a fundamental group of a graph.  Namely, for a graph $G$, we specify a base vertex $v_*$, and define a based loop as a map $\phi:I_n\rightarrow G$ where $I_n$ is a path on vertices $0,...,n$, and $\phi$ satisfies $\phi(0) = \phi(n) = v_*$.  Here, the map $\phi$ is a \emph{graph map}, meaning that for $x\sim y$, either $\phi(x) \sim \phi(y)$ or $\phi(x)=\phi(y)$.  Two loops are considered equivalent if there is a C-homotopy between them, where homotopy is defined in a way anologous to homotopy of algebraic topology.  The exact definition of this is not needed here, but details kind be found in \cite{Grigoryan_homotopy}.  We will make use of the following result from \cite{Grigoryan_homotopy} to determine when two loops are equivalent.  For our purposes, we can take this as the definition of C-homotopy.  To state this, we need the following terminology: given a loop $\phi:I_n\rightarrow G$, the \emph{word} of $\phi$, denoted $\theta_\phi$ is the sequence $v_0,...,v_n$ with $v_i = \phi(i)$ for $i=0,...,n$.  

\begin{theorem}[Theorem 4.13 of \cite{Grigoryan_homotopy}]\label{thm:homotopy}
Two based loops $\phi:I_n\rightarrow G$ and $\psi:I_m\rightarrow G$ are C-homotopic if and only if the word $\theta_\psi$ can be obtained from $\theta_\phi$ by a finite sequence of the following transformations and their inverses:
\begin{enumerate}
\item $...abc...\mapsto ac$ where $(a,b,c)$ are vertices forming a triangle in $G$ (and the ... denotes the unchanged part of the word);
\item $...abc...\mapsto ...adc...$ where $(a,b,c,d)$ forms a square in $G$;
\item $...abcd...\mapsto ...ad...$ where $(a,b,c,d)$ is a square in $G$;
\item $...aba...\mapsto ...a...$ if $a\sim b$;
\item $...aa...\mapsto ...a...$.
\end{enumerate}
\end{theorem}

One interpretation of this is that triangles, squares, and single edges are \emph{contractible} subgraphs of a graph.  

The set of all equivalence classes of loops in $G$ forms a group called the \emph{fundamental group} of $G$, denoted $\pi_1(G)$.  The group operation is concatenation of loops, the identity element is the trivial loop that maps all vertices to the base vertex, and the inverse of a loop is the loop traversed in reverse order.  See \cite{Grigoryan_homotopy} for details of why this is well-defined and forms a group.  

\subsection{Curvature bounds in graphs}
For a graph $G=(V,E)$, the \emph{graph Laplacian} is the operator $\Delta$ on the space of functions $f:V\rightarrow \R$ given by
\[
\Delta f(x) = \sum_{y\sim x}(f(y) - f(x)).
\]
The \emph{Bakry-\'Emery operators} are defined via
\begin{align*}
\Gamma(f,g) &:= \frac12\left(\Delta(fg) - f\Delta g - g\Delta f\right)\\
\Gamma_2(f,g) &:= \frac12\left(\Delta\Gamma(f,g)-\Gamma(f,\Delta g) - \Gamma(g,\Delta f)\right).
\end{align*}
We write $\Gamma(f) :=\Gamma(f,f)$ and $\Gamma_2(f) = \Gamma_2(f,f)$.

\begin{definition}[Bakry-\'Emery Curvature]
A graph $G$ is said to satisfy the \emph{curvature dimension inequality} $CD(K,n)$ for some $K\in\R$ and $n\in(0,\infty]$ if for all $f$,
\[
\Gamma_2(f) \ge \frac1n (\Delta f)^2 + K\cdot \Gamma(f).
\]
\end{definition}

We state a diameter bound in terms of curvature, similar to the Bonnet-Myers theorem from geometry, proven for graphs in \cite{liu2016bakry}.  A similar result is found in \cite{FathiYan2018}.

\begin{theorem}[Bonnet-Myers Theorem, Corollary 2.2 of \cite{liu2016bakry}]\label{thm:diam_bound}
Let $G$ be a graph satisfying $CD(K,\infty)$ for some $K>0$, and with maximum degree $Deg_{max}$.  Then
\[
diam(G) \le \frac{2Deg_{max}}{K}.
\]
\end{theorem}


\section{Curvature and coverings}\label{sec:CurvCover}

\begin{theorem}\label{thm:CurvCover}
Suppose a finite graph satisfies $CD(K,\infty)$ for some $K>0$. Then, there exists no infinite covering of $G$ preserving all 3- and 4-cycles.
\end{theorem}

\begin{proof}
Suppose there exists an infinite covering $(\widetilde G, \Psi)$ of $G$ preserving all 3- and 4- cycles. Then, $\widetilde G$ is locally isomorphic to $G$ and thus satisfying the same curvature bound $CD(K,\infty)$. We observe that $\widetilde G$ has bounded vertex degree $Deg$.
Now, Theorem \ref{thm:diam_bound} implies $diam(\widetilde G) \leq \frac {2Deg}K$ and thus finiteness of $\widetilde G$. This is a contradiction and therefore proves that there is no infinite covering of $G$ preserving 3- and 4-cycles.
\end{proof}

\section{Coverings and gain graphs}\label{sec:CoverCircuit}

For gain graphs, there is a natural construction of a covering of the graph that is derived from the gain function.  This construction is given in \cite{GrossTucker77} and is a variant on a construction from \cite{Gross74}.

\begin{definition}
Let $G=(V,E)$ be a graph
\begin{enumerate}
\item Let $\phi$ be a gain function on $G$ into group $\Gamma$. We define the \emph{ordinary derived graph} (or just \emph{derived graph}), denoted $G^{\phi} = (V^{\phi},E^{\phi})$ as 
\begin{align*}
V^{\phi} &= V\times \Gamma\\
E^\phi &= \left\{ \{(u,g),(v,g\phi(uv))\}: uv\in E, g\in\Gamma\right\}.
\end{align*}
\item For $\phi$ a gain function into the symmetric group $S_n$, $\sigma_{uv} = \phi(uv)$, we define the \emph{permutation derived graph} $G^\sigma = (V^\sigma,E^\sigma)$ as
\begin{align*}
V^\sigma &= V\times \{1,...,n\}\\
E^\sigma &= \left\{\{(u,i),(v,\sigma_{uv}(i))\}:uv\in E, i=1,...,n\right\}.
\end{align*}
\end{enumerate}
\end{definition}

Note that, as pointed out in \cite{GrossTucker77}, a permutation derived graph is \emph{not} simply an ordinary derived graph where the associated group is the symmetric group.  Indeed, the latter would have $n!\cdot|V|$ vertices, while the permutation derived graph has $n|V|$ vertices.  Also, in the definition, we allow $n=\infty$, and identify $S_\infty$ with the infinite symmetric group $Sym(\Z)$, in which case $G^\sigma$ is an infinite graph.  

There is a natural projection $\Psi:G^{\phi}\rightarrow G$ given by $\Psi((u,g)) = u$.  It is clear, as noted in \cite{GrossTucker77}, that $\Psi$ is a covering map, so that $(G^\phi,\Psi)$ is a covering of $G$ with $|\Gamma|$ sheets.  Similarly, $G^\sigma$ is a covering with $n$ sheets.  

We give a slight variation of our definition of the order of a cycle in a gain graph when the group is the symmetric group $S_n$.  For a gain graph where each edge has an associated permutation $\sigma_{uv}$, and for a circuit $C$ of the graph, we denote by $\sigma_C$ the composition of all the permutations going around the cycle.  For an element $i\in \{1,...,n\}$ we denote
\[
o_i(C) = \min\{k : \sigma_C^k(i) = i\}.
\]
Note that for $i\neq j$, $o_i(C)$ may not be equal to $o_j(C)$, depending on the action of $\sigma_C$ on $i$ and $j$.

\begin{lemma}\label{lem:perm_lift_cycle}
Let $G$ be a gain graph with group $S_m$ and corresponding permutation derived covering $(G^\sigma,\Psi)$.
Given a circuit $C$ of length $n$ of $G$, the pre-image $\Psi^{-1}(C)$ consists of a collection of vertex disjoint circuits $\{\widetilde C\}$, where if $\widetilde C$ contains the vertex $(x_1,i)$ for some $i\in\{1,...,m\}$, then the length of $\widetilde C$ is $o_i(C)\cdot n$.  In the case $o_i(C)$ is infinite, $\Psi^{-1}(C)$ contains an infinite path.
\end{lemma}
\begin{proof}
Let $C = (x_1,...,x_n)$ be a circuit of $G$, and fix $i\in\{1,...,m\}$.  Set $i_1 = i$, and define $i_{j+1} = \sigma_{x_jx_{j+1}}(i_j)$ where the index $j$ on the $x$'s is taken $\pmod{n}$.  Then since $x_jx_{j+1} \in E$ for all $j$, then $\{(x_j,i_j),(x_{j+1},i_{j+1})\}\in E^\sigma$ for all $j$ by definition.  Now we ask, when (if ever) does the sequence $\widetilde C = \left((x_1,i_1),(x_2,i_2),....\right)$ return to its starting point at $(x_1,i_1)$.  Clearly, for this to be the case, the index $j$ satisfies $j\equiv 1\pmod{n}$, and every time $j$ becomes $1 \pmod{n}$, the associated element gets mapped by $\sigma_C$.  So the sequence comes back to $(x_1,i_1)$ when the $x_j$ have come back to $x_1$ $o_i(C)$ times.  The sequence cannot intersect itself at any earlier point by minimality of $o_i(C)$.  Therefore clearly $\widetilde C$ is a circuit of length $o_i(C)\cdot n$.  If there are multiple distinct circuits in $\Psi^{-1}(C)$, it is clear that they are vertex disjoint by the definition of $G^\sigma$.  
\end{proof}

A slight modification of the above proof can be applied to ordinary derived graphs (not permutation derived) for any group.

\begin{lemma}\label{lem:derived_lift_cycle}
Let $(G,\phi,\Gamma)$ be a gain graph, with $(G^\phi,\Psi)$ the corresponding ordinary derived covering.  Given a circuit $C$ of length $n$ of $G$, the pre-image $\Psi^{-1}(C)$  consists of a collection of vertex disjoint circuits $\{\widetilde C\}$ where each $\widetilde C$ is of length $o_\phi(C)\cdot n$. In the case $o_\phi(C)$ is infinite, $\Psi^{-1}(C)$ contains an infinite path.
\end{lemma}

We remark that for the derived cover $(G^{\phi},\Psi)$, every circuit in $\Psi^{-1}$ has the same length, $o_\phi(C)\cdot n$, but for the permutation derived cover $(G^\sigma,\Psi)$, $\Psi^{-1}(C)$ may contain circuits of different lengths, if the permutation $\sigma_C$ has different orders for different elements of $\{1,...,m\}$. 

\begin{corollary}\label{cor:cover_pres_balance}
Let $(G,\phi,\Gamma)$ be a gain graph with $(G^\phi,\Psi)$ the associated ordinary derived covering, and if $\Gamma$ is some permutation group, let $(G^\sigma,\Psi)$ be the associated permutation derived covering.  Then for both $G^\phi$ and $G^\sigma$, if $\Bu$ is a collection of circuits of $G$, then $\phi$ is balanced on $\Bu$ if and only if the covering preserves $\Bu$.
\end{corollary}
\begin{proof}
For the permutation derived covering $G^\sigma$, a cycle $C$ of length $n$ is balanced if and only of $\sigma_C$ is the identity, which holds if and only if $\sigma_C(i) = i$ for all $i\in\{1,...,m\}$, in other words, $o_i(C) = 1$ for all $i$, so that $\Psi^{-1}(C)$ is a collection of vertex disjoint cycles of length $n$ by Lemma \ref{lem:perm_lift_cycle}.  This is the definition of $C$ being preserved under the cover.

A similar argument works for $G^\phi$.  
\end{proof}


The following theorem from \cite{GrossTucker77} shows that permutation derived coverings account for all possible coverings of a graph.

\begin{theorem}[Theorem 2 of \cite{GrossTucker77}]\label{thm:all_cover_perm}
Let $(\widetilde G,\Psi)$ be a covering of $G$ with $m$ sheets.  Then there is a gain function into the symmetric group $S_m$ assigning each edge $uv$ a permutation $\sigma_{uv}\in S_m$ on $G$ such that the permutation derived graph $G^\sigma$ is isomorphic to $\widetilde G$.  
\end{theorem}

We remark here that $m$ need not be finite, in which case the associated symmetric group can be identified with $Sym(\Z)$.

\begin{theorem}\label{thm:StrongGenCover}
Let $G=(V,E)$ be a connected graph and let $\mathcal B$ be a set of circuits. T.f.a.e:
\begin{enumerate}
\item There exists no  non-trivial covering of $G$ preserving $\mathcal B$
\item  The circuit set $\mathcal B$ is a $\Gamma$ circuit generator of $G$ for all groups $\Gamma$.
\item The circuit set $\mathcal B$ is  a $Sym(\Z)$ circuit generator of $G$.
\item The circuit set $\mathcal B$ is  a $\Gamma(G,\mathcal B)$ circuit generator of $G$.
\end{enumerate}
\end{theorem}

\begin{proof}
We begin by proving 1 $\implies$ 2 by contraposition.  Suppose that there is some group $\Gamma$ for which $\Bu$ is not a $\Gamma$ circuit generator.  That means there is a gain $\phi:\vec E \rightarrow \Gamma$ that is balanced on $\Bu$, but that is unbalanced on some circuit, say $C$ of $G$.  Then $o_\phi(C) > 1$.  Consider the derived covering $(G^\phi, \Psi)$.  By Corollary \ref{cor:cover_pres_balance}, this covering preserves $\Bu$.  By Lemma \ref{lem:derived_lift_cycle}, since $o_\phi(C) >1$, $\Psi^{-1}(C)$ contains a circuit of length strictly larger than $C$.  This implies that the covering is non-trivial.

That 2 $\implies$ 3 is immediate.

Now we prove 3 $\implies$ 1 by contraposition.  Suppose that there is a non-trivial covering of $(\widetilde G,\Psi)$ of $G$ that preserves $\Bu$, and suppose it has $m$ sheets. Then by Theorem \ref{thm:all_cover_perm}, this covering can be realized as a permutation derived covering $(G^\sigma,\Psi)$ for some assignment $\sigma_{uv}\in S_m$ for all $uv\in E$.  Since the covering preserves $\Bu$, then every circuit from $\Bu$ is balanced under this assignment by Corollary \ref{cor:cover_pres_balance}.  We need to find some circuit of $G$, not in $\Bu$, that is unbalanced.  Since $(G^\sigma,\Psi)$ is a non-trivial covering, there exist two vertices, call them $\widetilde x_1$ and $\widetilde x_k$ in the same connected component of $\widetilde G$ such that $\Psi(\widetilde x_1) = \Psi(\widetilde x_k)$.  Since these belong to the same component, there is a path connecting $\widetilde x_1$ to $\widetilde x_k$.  Since the covering preserves edges, and  $\Psi(\widetilde x_1) = \Psi(\widetilde x_k)$, the image of this path is a circuit of $G$.  Call this circuit $C$, and suppose it has length $n$, and let $x_1 =  \Psi(\widetilde x_1) = \Psi(\widetilde x_k)$.  By Lemma \ref{lem:perm_lift_cycle}, the pre-image $\widetilde C$ of $C$ is a circuit of length $o_i(C)\cdot n$ for some $i$.  But since the pre-image of $x_1$ has at least two vertices, then $\widetilde C$ has strictly more vertices than $C$, implying $o_i(C) >1$.  Thus the permutation $\sigma_C$ cannot be the identity, so $C$ is not balanced.  So $\Bu$ is not a $S_m$ circuit generator, for any $m$.  Thus $\Bu$ is not a $Sym(\Z)$ circuit generator either, by Proposition \ref{prop:finitelyGenerated}.

It is clear that 2 $\implies$ 4, and 4 $\implies$ 2 by Proposition \ref{prop:universal}.
\end{proof}


\section{Gain graphs and cycle bases}\label{sec:BasisGain}

\begin{theorem}
Every weakly fundamental cycle basis is a combinatorial circuit generator.
\end{theorem}

\begin{proof}
Let $\{C_1,\ldots,C_r\}$ be a weakly fundamental cycle basis. Due to the weakly fundamental property, we can assume without obstruction that $C_r$ contains an edge $e=x_1x_2$ not contained in all other $C_i$.
We assume by induction that the theorem holds true for all cyclomatic numbers smaller than $r$.
Let $C$ be a circuit. Due to induction, we can assume that $C$ contains $e$ since otherwise $C$ can be be represented by $C_1,\ldots,C_{r-1}$ and we can delete $e$ and $C_r$ from the graph and decrease the cyclomatic number.
We aim to write
\begin{align}
C=(\ldots (C_r \oplus K_1) \oplus \ldots )\oplus K_n   \label{eq:weaklyFundamentalCombinatorial}
\end{align}
with $K_i \in \langle C_1,\ldots,C_{r-1} \rangle$.
This would prove the theorem since due to induction over the cyclomatic number, we can assume that $K_1,\ldots,K_n$ are in the closure of $\mathcal B$ under $\oplus$.

We inductively define $K_k$ and
\[
C_r^k := (\ldots (C_r \oplus K_1) \oplus \ldots )\oplus K_k
\] 
for all $k$. We start with $C_r^0=C_r$.
Write
\[
C_r^k = (x_1,\ldots,x_N).
\]
Now, we can write 
\[
C=(x_1,\ldots, x_K, y_1,\ldots,y_L,x_M,\ldots)
\]
with $y_i \notin C_r^k$, and after $x_M$ are some further elements of $C$ that we need not specify.
Define
\[
K_{k+1} := (y_1,\ldots,y_L,x_M,\ldots,x_K).
\]
Observe that $K_{k+1}$ does not contain the edge $x_1x_2$, so $K_{k+1}$ belongs to the closure of $\Bu$ under $\oplus$ by the induction hypothesis. 
Note also that $C_r^k$ and $K_{k+1}$ belong to a theta graph whose paths are given by $p_1=x_K,x_{K+1},...x_{M-1},x_M$; $p_2=x_K,y_1,...,y_L,x_M$; and $p_3 = x_M,...x_N,x_1,...,x_K$.
Define
\[
C_r^{k+1} := C_r^k \oplus K_{k+1} = (x_1,\ldots, x_K, y_1,\ldots,y_L,x_M,\ldots,x_N).
\]
Observe that $C_r^{k+1}$ and $C$ share a longer path than $C_r^k$ and $C$ do. Therefore, this process will terminate yielding (\ref{eq:weaklyFundamentalCombinatorial}). This finishes the proof.
\end{proof}



\begin{theorem}
Every combinatorial circuit generator is a $\Gamma$-circuit generator for all groups $\Gamma$.
\end{theorem}
\begin{proof}
Let $\mathcal B$ a combinatorial circuit generator and define inductively $\mathcal B_0:= \mathcal B$ and
\[
\mathcal B_{k+1} := \mathcal B_k \cup(\mathcal B_k \oplus \mathcal B_k) = \mathcal B_k \cup  \{C_1 \oplus C_2:  C_1,C_2 \in \mathcal B_k\}.
\]
Since $\mathcal B$ is a combinatorial circuit generator, we have $\bigcup_{k \in \N} \mathcal B_k = \mathcal S(G)$.
Suppose $\mathcal B$ is not a $\Gamma$-circuit generator for some group $\Gamma$.
Let $K$ be the minimal number such that there exists a circuit $C \in \mathcal B_k$ and $\phi \in \Phi(G,\Gamma)$ such that $\phi$ is balanced on $\mathcal B$ but not on $C$. Due to minimality, $\phi$ is also balanced on $B_{k-1}$.
Since $C \in \mathcal B_k \setminus \mathcal  B_{k-1}$, we can write $C = C_1 \oplus C_2$ for some $C_1,C_2 \in \mathcal B_{k-1}$, where $C$, $C_1$, and $C_2$ all belong to a single theta graph in $G$.
Thus, we can write $C_i = (p_i, p_3)$  for some paths $p_1,p_2,p_3.$ Moreover, $C=(p_1,p_2^{-1})$ where $p_2^{-1}$ is the inverse of the path $p_2$.
Due to induction, we have $\phi(p_i)\phi(p_3) = e_\Gamma$ and thus $\phi(p_1)\phi(p_2^{-1}) = e_\Gamma$ showing that $\phi$ is balanced on $C$. This is a contradiction and proves that $\mathcal B$ is a $\Gamma$-circuit generator. This finishes the proof. \end{proof}

\begin{theorem}\label{thm:CycleCircuitField}
Let $\F$ be a field with additive group $\Gamma$.
Let  $\mathcal B = \{C_1, . . . , C_r\} \subset \mathcal C(G,\F)$ be a cyclomatic circuit set of a graph $G$. T.f.a.e.:
\begin{enumerate}
\item
$\mathcal B$ is a $\F-$cycle basis.
\item $\mathcal B$ is a $\Gamma$-circuit generator.
\end{enumerate}
\end{theorem}

\begin{proof}
1 $\Rightarrow$ 2:
We aim to show that every gain function $\phi$ is balanced on all cycles when assuming that $\phi$ is balanced on $\mathcal B$.
Now, $\phi$ is balanced on $C$ if and only $\phi(C)=0$. Since $\phi$ is linear and $\phi(C)=0$ for all $C\in \mathcal B$ due to assumption, we infer that $\phi(C)=0$ for all $C \in \SPAN(\mathcal B) = \mathcal C(G,\F)$ since we assume that $\mathcal B$ is a $\F$-cycle basis of $G$.

2 $\Rightarrow$ 1:
We indirectly prove the claim. 
Assume $\mathcal B$ is not a $\F$-cycle basis. Then, there exists a basis $\widetilde {\mathcal B}$ and $C_0 \in \mathcal {\widetilde B}$ s.t. $\mathcal B \subset \SPAN(\mathcal {\widetilde B} \setminus \{ C_0\})$.
The matrix $M(\widetilde {\mathcal B},\F)$ is a $r \times |E|$ matrix with full rank $r$. 
Hence, the multiplication with the gain functions 
$M(\widetilde {\mathcal B},\F) : \Phi(G,\F) \to \F^{\widetilde {\mathcal B}}$
 is surjective. In particular, there exists $\phi \in \Phi(G,\F)$ s.t. for $C \in \mathcal{\widetilde B}$,
 \[
 \phi(C)=\left[M(\widetilde {\mathcal B},\F) \phi\right](C)= \begin{cases}
 1 &: C=C_0\\
 0 &: C \in \mathcal {\widetilde B} \setminus \{ C_0\}
 \end{cases}.
 \]
 This implies $\phi(C)=0$ for all $C \in \mathcal B$ and $\phi(C_0)=1$ which proves that $\mathcal B$ is not a $\Gamma$ circuit generator.
 This finishes the proof.
\end{proof}

\Hidden{
\begin{corollary}\label{cor:CurvBasis}
If a graph $G$ satisfies $CD(K,\infty)$ for some $K>0$, and if $\F$ is a field with characteristic 0, then there is an $\F$-cycle basis consisting of only 3- and 4-cycles.  
\end{corollary}
\begin{proof}
Let $\mathcal B$ be the set of 3- and 4-cycles of $G$.
By Theorem \ref{thm:CurvCover}, there is no infinite cover of $G$ preserving preserving $\mathcal B$.  Then take $\Gamma$ to be the additive group of the field $\F$.  By Theorem \ref{thm:CoverWeakGen}, the set $\mathcal B$ is a weak $\Gamma$-circuit generator of $G$.  Since $\F$ has characteristic 0, no non-trivial element has finite order, so in fact $\mathcal B$ is a strong $\Gamma$-circuit generator of $G$.  Then by Theorem \ref{thm:CycleCircuitField}, $\mathcal B$ contains a $\F$-cycle basis.
\end{proof}
} 

Since all proper finitely generated subgroups of $(\Q,+)$ are isomorphic to $\Z$ we immediately obtain the following corollary by using Theorem~\ref{thmCycleDet}.

\begin{corollary}\label{cor:QZcircuitGenerators}
Let $G$ be a finite graph and let $\mathcal B \subset \mathcal S(G)$ be a cyclomatic circuit set. T.f.a.e.:
\end{corollary}
\begin{enumerate}
\item
$\det \mathcal B \neq 0$.
\item
$\mathcal B$ is a $\Q$-circuit generator.
\item
$\mathcal B$ is a $\Z$-circuit generator.
\end{enumerate}

\begin{corollary}\label{cor:qnDetAbel}
Let $n \in \N$ and let $q$ be a prime number. Let $\Gamma$ be the cyclic group with $q^n$ elements and let $G$ be a graph with a cyclomatic circuit set $\mathcal B \subset \mathcal S(G)$. T.f.a.e.:
\begin{enumerate}
\item $\det \mathcal B \not\equiv 0 \mod q$.
\item $\mathcal B$ is a $\Gamma$ circuit generator.
\end{enumerate}
\end{corollary}
\begin{proof}
The implication 2 $\Rightarrow$ 1 follows from Theorem~\ref{thmCycleDet}, Theorem~\ref{thm:CycleCircuitField} and Proposition~\ref{prop:products} since the additive group of $\F_q$ is a subgroup of $\Gamma$.

We next prove 1 $\Rightarrow$ 2.
We say $\mathcal B=\{C_1,\ldots,C_r\}$.
We canonically identify the elements of $\Gamma$ with $\{0,\ldots,q^n-1\} \subset \Q$ via a function $\eta : \Gamma \to \Q$. 
Since $\det \mathcal B \not \equiv 0 \mod q$, the matrix
$M(\mathcal B,\Q,T)$ is invertible for a spanning tree $T$ and every circuit $C \in \mathcal C(G,\Q)$ can uniquely be written as
\[
C = \frac{\lambda_1}{\det \mathcal B} C_1 + \ldots + \frac{\lambda_r}{\det \mathcal B}  C_r
\]
with integers $\lambda_1,\ldots,\lambda_r$.
Suppose $\mathcal B$ is not a $\Gamma$ circuit generator.
Then there exists $\phi \in \Phi(G,\Gamma)$ that is balanced on $\mathcal B$ but not on some $C \in \mathcal S(G)$.
Observe $\eta \circ \phi \in \Phi(G,\F)$ and $\phi$ is balanced on a circuit $C$ if and only if $(\eta \circ \phi) (C) \equiv 0 \mod q^n$.
Therefore we can write, $(\eta \circ \phi)(C_i)= c_i q^n$ for integers $c_i$ and $i=1,\ldots,c_r$. This implies
\[
(\eta \circ \phi)(C)\cdot \det \mathcal B = (\lambda_1 c_1 + \ldots +  \lambda_r c_r)q^n \in q^n \Z.
\]
Since $\det \mathcal B \not\equiv 0 \mod q$ and since $q$ is prime, we obtain
$(\eta \circ \phi)(C) \equiv 0 \mod q^n$ which is equivalent to balance of $\phi$ on $C$.
This contradicts the assumption that $\mathcal B$ is not a $\Gamma$ circuit generator.
This finishes the proof.
\end{proof}

Using the above theorem, we can fully characterize $\Gamma$ circuit bases for Abelian groups $\Gamma$

\begin{corollary}
Let $\Gamma$ be an Abelian group and
let  $\mathcal B = \{C_1, . . . , C_r\} \subset \mathcal S(G)$ be a cyclomatic circuit set of a graph $G$. T.f.a.e.:
\begin{enumerate}
\item
$g^{\det \mathcal B} \neq e_\Gamma$ for all $g \in \Gamma \setminus \{e_\Gamma\}$
\item $\mathcal B$ is a $\Gamma$-circuit generator.
\end{enumerate}
\end{corollary}

\begin{proof}
We start proving 2 $\Rightarrow$ 1.
First suppose $\det \mathcal B \neq 0$.
Suppose $g^{\det \mathcal B} = e_\Gamma$ for some $g \in \Gamma\setminus \{e_\Gamma\}$. Then, we can assume without obstruction that $g$ has prime order $q$ and $\det \mathcal B \equiv 0 \mod q$.
Theorem~\ref{thmCycleDet} and Theorem~\ref{thm:CycleCircuitField} imply that $\mathcal B$ is not a $\langle g\rangle$ circuit generator. Since $\langle g \rangle$ is a subgroup of $\Gamma$, Proposition~\ref{prop:products} implies that $\mathcal B$ is not a $\Gamma$ circuit generator.
Now suppose $\det \mathcal B = 0$. If there exists  $g \in \Gamma\setminus \{e_\Gamma\}$ with finite order, we can proceed as in the case above.
Otherwise, there exists an element $g \in \Gamma$ with infinite order and therefore $(\Z,+) \cong   \langle g\rangle \subset \Gamma$.
Corollary~\ref{cor:QZcircuitGenerators} yields that $\mathcal B$ is not a $\langle g \rangle$ circuit basis since $\det \mathcal B = 0$. Since $\langle g \rangle$ is a subgroup of $\Gamma$, Proposition~\ref{prop:products} also yields that $\mathcal B$ is not a $\Gamma$ circuit basis. This finishes the proof of  2 $\Rightarrow$ 1.

We now prove 1 $\Rightarrow 2$. 
W.l.o.g., $\Gamma$ is not the one element group and therefore, we can assume $\det \mathcal B \neq 0$.
Due to Proposition~\ref{prop:finitelyGenerated} we can assume without obstruction that $\Gamma$ is finitely generated.
Therefore, $\Gamma$ is isomorphic to $\Z^n \oplus  \Z_{q_1} \oplus \ldots \oplus \Z_{q_t}$ where $q_1,\ldots,q_t$ are powers of prime numbers and $\Z_q$ is the cyclic group with $q$ elements.
Due to Proposition~\ref{prop:products} it suffices to show that
$\mathcal B$ is a $\Z$ circuit generator and a $\Z_{q_i}$ circuit generator for $i=1,\ldots,r$. Corollary~\ref{cor:QZcircuitGenerators} and $\det \mathcal B \neq 0$ implies that $\mathcal B$ is a $\Z$ circuit generator.
Let $i \in \{1,\ldots r\}$. 
We know $q_i=p^n$ for some prime $p$ and we know $\Z_{q_i}$ has order $p$.
Therefore, assertion 1 of the theorem implies $\det \mathcal B \not \equiv 0 \mod p$. Now, Corollary~\ref{cor:qnDetAbel} applied to $\Z_{q_i}$ yields that $\mathcal B$ is a $\Z_{q_i}$ circuit generator. Since $i$ is arbitrary, this shows that $\mathcal B$ is a $\Gamma$ circuit generator.
This finishes the proof.
\end{proof}

\section{Cycle space and homology}\label{sec:BasisHomology}
We will be dealing with undirected graphs $G=(V,E)$, so when considering the homology groups, we will view $G$ as a directed graph in which each edge corresponds to two directed edges, one in each direction.

\begin{theorem}\label{thm:HomologyCycleSpace}
Let $\F$ be a field with characteristic not equal to 2, and let $\Cy(G,\F)$ denote the $\F$-cycle space of $G$.  Let $TS$ denote the subspace of $\Cy(G,\F)$ that is generated by all simple triangles and squares of $G$. Then
\[
H_1(G,\F) \cong \Cy(G,\F)/TS.
\]
\end{theorem}

This section will be devoted to proving this theorem.

Recall that by definition,
\[
H_1(G,\F) = Ker~\de_1 / Im~\de_2
\]
where $\de_1$ denotes the boundary operator on 1-paths and $\de_2$ the boundary operator on 2-paths.

Observe that the space $\Omega_1$ of 1-paths can be naturally identified with the space of functions from the edge set to the field $\F$; that is
\[
\Omega_1 \cong \{\phi:\overrightarrow{E}\rightarrow\F\}.
\]
This space can naturally be decomposed: define \begin{align*}
\Omega_+ &= \{\phi\in \Omega_1 : \phi(xy) = \phi(yx) \text{ for all } x,y\}\\ \Omega_- &= \{\phi\in\Omega_1 : \phi(xy) = -\phi(yx) \text{ for all } x,y\}.\end{align*}  
Then it is clear that 
\[
\Omega_1 = \Omega_+ \oplus \Omega_-.
\]

\begin{lemma}\label{lem:ker}
$Ker~\de_1 \cong \Omega_+\oplus\Cy(G,\F).$\end{lemma}
\begin{proof}
The action of $\de_1$ on $\Omega_1$ can be given as
\[
\de_1\phi = \sum_{x,y}\phi(xy)(e_y-e_x) = \sum_x\sum_y(\phi_{yx}-\phi_{xy})e_x.
\]
Therefore $\phi\in Ker~\de_1$ if and only if
\[
\sum_{y\sim x}(\phi(yx) -\phi(xy)) =0 \text{ for all } x.
\]
In terms of the direct sum decomposition above, this then yields
\begin{align*}
Ker~\de_1 = &\left\{\phi: \phi(xy) = \phi(yx) \text{ for all } x,y\right\}\\&\oplus \left\{\phi : \phi(xy) = -\phi(yx) \text{ for all } x,y \text{ and } \sum_{y\sim x}\phi(xy) =0 \text{ for all } x\right\}.
\end{align*}
The first term is clearly $\Omega_+$ and the second is, by definition, the cycle space $\Cy(G,\F)$.  This gives the lemma.
\end{proof}  

\begin{lemma}\label{lem:im}
$Im~\de_2 \cong \Omega_+\oplus TS$.  \end{lemma}
\begin{proof}
First, any element of $\Omega_+$ can be written
\[
\sum_{x,y}\phi(xy)(e_{xy}+e_{yx}) = \de_2\left(\sum_{x,y}\phi(xy)e_{xyx}\right),
\]
so $\Omega_+\subset Im~\de_2$.  It remains to show that the portion of $Im~\de_2$ that lies in $\Omega_-$ is equal to $TS$.  That is, we must show that any element from $\Cy(G,\F)$ is in $Im~\de_2$ if and only if it is in the space spanned by triangles and squares.

Suppose $xyz$ is a triangle of $G$.  Consider first any triangle, that is, some $\phi$ such that $\phi(xy) = \phi(yz) = \phi(zx)=-\phi(yx)=-\phi(zy)=-\phi(xz)$ and is 0 on all other edges.  Note that
\begin{align*}
\de_2(\phi(xy)(e_{xyz}-e_{zyx})& = \phi(xy)e_{yz}-\phi(xy)e_{xz}+\phi(xy)e_{xy} - \phi(xy)e_{yx} + \phi_{xy}e_{zx} - \phi(xy)e_{yx}\\
&= \phi(xy)e_{xy}+\phi(yz)e_{yz}+\phi(zx)e_{zx}+\phi(yx)e_{yx}+\phi(zy)e_{zy}+\phi(xz)e_{xz}
\end{align*}
which is the triangle $\phi$. So any triangle is contained in $Im~\de_2$.

Similarly if $\phi$ is an square $xyzw$ of $G$, i.e. $\phi(xy) = \phi(yz) = \phi(zw) = \phi(wx) = -\phi(yx) = -\phi(zy) = -\phi(wz) = -\phi(xw)$ and is 0 elsewhere, then in a similar way, it can be verified that
\begin{align*}
\de_2(\phi(xy)(e_{xyz} - e_{xwz}-e_{zyx}+e_{zwx})) &=\phi.
\end{align*}
Therefore any square is in $Im~\de_2$ as well.  It follows that the space $TS\subset Im~\de_2$.

Conversely, we must show $Im~\de_2\subset \Omega_+\oplus TS$.  We already know $Im~\de_2\subset Ker~\de_1 = \Omega_+\oplus\Cy(G,\F)$, so we will be done if we can show that any $\phi\in Im~\de_2 \cap \Cy(G,\F)$ belongs to $TS$.  Since $\phi\in Im~\de_2$ we can write \begin{align*}
\phi &= \de_2\left(\sum_{xyz}a_{xyz}e_{xyz}\right)\\
&= \sum_{xyz}a_{xyz}(e_{yz}-e_{xz}+e_{xy}),
\end{align*}
where the sum is taken over $\de$-invariant allowed paths $xyz$ of $G$.  Since $\Omega_2$ consists only of $\de$-invariant allowed elements, $a_{xyz}$ is non-zero only for allowed paths $xyz$, implying that $xy$ and $yz$ are edges of $G$, and that either $xz$ is an edge of $G$, or else the $e_{xz}$ term cancels in the sum.  We will therefore split the above sum into two parts,
\[
\sum_{\substack{xyz\\xz\in E(G)}}a_{xyz}(e_{yz}-e_{xz}+e_{xy}) + \sum_{\substack{xyz\\xz\not\in E(G)}}a_{xyz}(e_{yz}-e_{xz}+e_{xy}).
\]
Since we are assuming $\phi \in \Cy(G,\F)$ (in particular, $\phi(xz) = -\phi(zx)$) then it is clear that the first term above is a linear combination of triangles.  

For the second term, since it is allowed, any $e_{xz}$ term must cancel.  Thus, for any $xyz$ for which $a_{xyz}$ is non-zero in the second sum, there must be some other allowed 2-path in which $e_{xz}$ shows up as a term.  Namely, there exists $w\neq y$ such that $xwz$ is allowed in $G$, and the coefficient $a_{xwz} = -a_{xyz}$.  It is clear then that the second sum is a linear combination of squares of $G$.  Thus we have shown that $Im~\de_2\cap \Cy(G,\F) \subset TS$, and we have shown the lemma.
\end{proof}

Lemma \ref{lem:ker} and Lemma \ref{lem:im} together immediately give the proof of Theorem \ref{thm:HomologyCycleSpace}.

With this, we are ready to prove Theorem \ref{hom_van}, which we restate as the following.

\begin{corollary}\label{cor:HomVanish}
If $G$ is a finite graph satisfying $CD(K,\infty)$ for some $K>0$ and if $\F$ is a field with characteristic 0, then 
\[
H_1(G,\F) = 0.
\]
\end{corollary}
\begin{proof}
Suppose by way of contradiction that $H_1(G,\F)$ is non-trivial.  Let $\Bu$ denote the collection of triangles and squares in $G$. Then by Theorem \ref{thm:HomologyCycleSpace}, the cycle space of $G$ is not generated by $\Bu$, and so by Theorem \ref{thm:CycleCircuitField}, $\Bu$ is not a $\Gamma$-circuit generator where $\Gamma$ is the additive group of $\F$.  Thus there is some gain function $\phi:\vec E\rightarrow \Gamma$ that is balanced on all triangles and squares, but is unbalanced on some other cycle, call it $C$.  Then we can construct the ordinary derived covering $G^\phi$ with projection $\Psi$ of Section \ref{sec:CoverCircuit}.  Since $\phi$ is not balanced on $C$ and since $\F$ has characteristic 0, then $o(\phi(C))=\infty$.  By Lemma \ref{lem:derived_lift_cycle}, $\Psi^{-1}(C)$ contains an infinite path, and by Corollary \ref{cor:cover_pres_balance}, $\Psi$ preserves $\Bu$.  But then $\Psi$ is an infinite covering of $G$ preserving all triangles and squares, so by Theorem \ref{thm:CurvCover}, $G$ cannot satisfy $CD(K,\infty)$ at every vertex.  This implies the result.

\Hidden{
Let $\mathcal B$ be the set of 3- and 4-cycles of $G$.
By Theorem \ref{thm:CurvCover}, there is no infinite cover of $G$ preserving preserving $\mathcal B$.  Then take $\Gamma$ to be the additive group of the field $\F$.  By Theorem \ref{thm:CoverWeakGen}, the set $\mathcal B$ is a weak $\Gamma$-circuit generator of $G$.  Since $\F$ has characteristic 0, no non-trivial element has finite order, so in fact $\mathcal B$ is a strong $\Gamma$-circuit generator of $G$.  Then by Theorem \ref{thm:CycleCircuitField}, $\mathcal B$ contains a $\F$-cycle basis.  Thus the $\F$ cycle space of $G$ has a basis consisting entirely of 3- and 4-cycles.  Then Theorem \ref{thm:HomologyCycleSpace} implies the result.
}
\end{proof}

Of course, the converse of Corollary \ref{cor:HomVanish} does not hold.  Indeed, it is well-known that in trees other than paths or the star on 4 vertices, there will typically be vertices with negative curvature (see, for instance \cite{cushing2016bakry}).  However, all trees have trivial first homology \cite{Grigoryan2}.

It is natural to ask if the hypotheses of Corollary \ref{cor:HomVanish} can be weakened, but still obtain trivial first homology.  Consider however the graph pictured below.
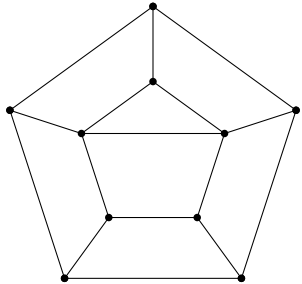
\tikzstyle{every node}=[circle,fill=black,inner sep=1pt]
\begin{center}
\begin{tikzpicture}
\draw \foreach \x in {18,90,...,306}
{
(\x:1) -- (\x+72:1) node{}
(\x:2) -- (\x+72:2) node{}
(\x:1) -- (\x:2) node{}
};
\draw (18:1)--(162:1);
\end{tikzpicture}
\end{center}
Computation shows that this graph has non-negative curvature at every vertex, and strictly positive curvature at some vertices.  However $dim~H_1(G,\F)=1$, since the outer 5-cycle is not generated by any 3- or 4-cycles.  This shows that we cannot weaken the hypothesis to simply non-negative curvature, even if there are some vertices with strictly positive curvature. 

Another result due to Bochner \cite{BochnerYano} is that a Riemannian manifold of non-negative curvature has finite-dimensional first homology group. We conjecture that this holds for graphs as well.

\subsection{A remark on clique homology}
Another commonly used notion of homology in graph theory is the \emph{clique homology} coming from the \emph{clique complex}, or \emph{flag complex} of the graph.  In this theory, the chain complex is
\[
\cdots C_n\rightarrow C_{n-1}\rightarrow\cdots\rightarrow C_1\rightarrow C_0\rightarrow0
\]
where $C_n$ is the space of all formal $\F$-linear combinations of $n$-cliques of the graph $G$. (Hence it is still the case that $C_1$ is all formal linear combinations of edges, and $C_0$ all formal linear combinations of vertices.)  The boundary map $\de$ of a clique is the sum of all its ``faces," viewing the graph as a cell complex with an $n$-cell filling each $n$-clique.  Then the clique homology groups are defined in the same way,
\[
H_n^{\text{clique}}(G,\F) = Ker~\de|_{C_n}/Im~\de|_{C_{n+1}}.
\]

Using techniques very similar to those used to prove Theorem \ref{thm:HomologyCycleSpace}, it is possible to prove an analogous theorem for clique homology.

\begin{theorem}\label{thm:CliqueHomologyCycleSpace}
Let $\F$ be a field with characteristic not equal to 2, and let $\Cy(G,\F)$ denote the $\F$-cycle space of $G$.  Let $T$ denote the subspace of $\Cy(G,\F)$ that is generated by all simple triangles of $G$. Then
\[
H_1^{\text{clique}}(G,\F) \cong \Cy(G,\F)/T.
\]
\end{theorem}

Hence for both the path and clique homology theories, the first homology group ``counts" cycles of the graph, but there are certain types of cycles ignored depending on the theory; clique homology does not see triangles, and path homology sees neither triangles nor squares.  

Observe in particular that the homology vanishing theorem for path homology, Corollary \ref{cor:HomVanish}, does not hold for the clique homology (a simple 4-cycle being a counterexample).  We take this as further evidence that the path homology is the more appropriate homology theory for graph theory.  

\section{Homotopy and fundamental groups for graphs}\label{sec:homotopy}
In this section we examine the fundamental group of a graph as defined by Grigoryan et. al. \cite{Grigoryan_homotopy}, and connect this to the group for the canonical gain graph defined previously.

Recall from Theorem \ref{thm:homotopy}, the fundamental group of a graph can be described as the group of equivalence classes of loops, where loops are equivalent if their corresponding words differ by a finite sequence of application of some rules.  These rules amount to triangles, squares, and trees being contractible.  We will generalize this notion.

\begin{definition}\label{def:pi}
Let $\Bu$ be a collection of circuits of a graph $G$. Define $\pi_1(G,\Bu)$ to be the group of equivalence classes of loops in $G$ where two loops $\phi:I_n\to G$ and $\psi:I_m\to G$ are considered equivalent if the word $\theta_\psi$ can be obtained from $\theta_\phi$ via a finite sequence of the following transformations and their inverses:
\begin{enumerate}
\item $...av_1...v_ib...\mapsto ...aw_1...w_jb...$ where $(a,v_1,...,v_i,b,w_j,...,w_1)$ is a circuit from $\Bu$;\label{rule:cycle}
\item $...aba...\mapsto ...a...$ if $a\sim b$;\label{rule:aba}
\item $...aa...\mapsto ...a....$\label{rule:aa}
\end{enumerate}
\end{definition}

Observe then that the fundamental group $\pi_1(G)$ is precisely $\pi_1(G,\Bu)$ where $\Bu$ is the collection of all triangles and squares in $G$.

Recall that, in Section \ref{sec:prelim}, we defined for a family of circuits $\Bu$ the canonical gain group $\Gamma(G,\Bu)$.  Now, for a fixed spanning tree $T$ of a graph $G$, we define the group $\Gamma(G,T,\Bu)$ via the presentation
\[
\Gamma(G,T,\Bu) =  \langle \vec E \mid T, \Bu \rangle.
\]

\begin{theorem}\label{thm:pi}
For any spanning tree $T$ of $G$, we have 
\[
\pi_1(G,\Bu) \cong \Gamma(G,T,\Bu).
\]
In particular, the group $\Gamma(G,T,\Bu)$ is independent of the spanning tree $T$ up to isomorphism.
\end{theorem}
\begin{proof}
Let $\Gamma := \langle \vec E \mid T \rangle$ where we identify the edge $xy$ with $(yx)^{-1}$.  Define a map $\phi:\Gamma\to \pi_1(G,\Bu)$ as follows.  First, given $g \in \Gamma$, choose the shortest representative word $e_1...e_k$ without spanning tree edges, then associate to this word the loop given by starting at the base point $v_*$, and taking the unique path from $v_*$ through $T$ to the starting point of $e_1$, then go to the endpoint of $e_1$, and take the unique path in $T$ from that vertex, to the starting vertex of $e_2$, continue in this manner until we reach the endpoint of $e_k$, and take the unique path in $T$ from there to $v_*$. Then $\phi(e_1...e_k)$ is the equivalence class of this loop in $\pi_1(G,\Bu)$.

Recall that $\Gamma(G,T,\Bu) = \langle \vec E \mid T, \Bu\rangle$,so by definition of a group presentation, is $\Gamma/\langle \Bu \rangle$ where $\langle \Bu\rangle$ is the normal closure of the set $\Bu$.  So what we need to show is that $\phi$ is a well-defined surjective group homomorphism whose kernel is the normal closure of $\Bu$.  Then we will be done by the first isomorphism theorem for groups.

Clearly, $\phi$ is  well-defined.
    To see that it is a homomorphism, consider $\phi(e_1...e_j)\phi(e_{j+1}...e_k)$.  Since $T$ is a spanning tree, the path in $T$ from the endpoint of $e_j$ to $v_*$, and from $v_*$ to the start of $e_{j+1}$ is equivalent to the path in $T$ from the end of $e_j$ to the start of $e_{j+1}$, possibly via application of the $...aba...\mapsto ...a...$ rule of the definition of $\pi_1(G,\Bu)$.  Thus $\phi(e_1...e_j)\phi(e_{j+1}...e_k) = \phi(e_1...e_k)$ as desired.  

To show that $\phi$ is surjective, suppose the sequence $v_*,v_1,...,v_k,v_*$ is the word of a loop in $G$. Then either $v_i=v_{i+1}$ or $(v_i,v_{i+1})$ is an edge of $G$.  If $v_i=v_{i+1}$, we can get rid of one of these via the $...aa...\mapsto ...a...$ rule. 
Due to the $...aba... \mapsto ...a...$, we can assume without obstruction that the loop is non-backtracking, i.e., $v_{i+2}\neq v_i$ for all $i$.
We consider the loop $v_*w_1...w_nv_*=\phi((v_*,v_1)(v_1,v_2)...(v_{k-1},v_k)(v_k,v_*))$.
Since between every two vertices within a spanning tree there exists a unique non-backtracking path connecting both, we infer $w_i=v_i$ and $n=k$ meaning that $\phi$ maps $g$ to the loop $v_*v_1...v_kv_*$. This proves surjectivity of $\phi$.

Now, to see that $\Bu\subset Ker(\phi)$, suppose that if $e_1,...,e_k = (u_1,u_2),...,(u_k,u_1)$ are edges of a circuit from $\Bu$ with vertices $u_1,...,u_k$.  Then $\phi (e_1...e_k)$ is a loop whose word is has the form $v_*...v_n u_1u_2...u_ku_1v_n...v_*$, and by rules \ref{rule:cycle} and \ref{rule:aba} of Definition \ref{def:pi}, so this word belongs to $Ker(\phi)$.

Finally, we wish to show that $Ker(\phi)$ is a subset of the normal closure of $\Bu$.  Let $w=e_1...e_k\in Ker(\phi)$. Then $\phi(w)$ is equivalent to the trivial loop, so this means that the word of the the loop $\phi(w)$ can be obtained form the trivial word $v_*$ from a sequence transformations using the rules in Definition \ref{def:pi}.  
Rules \ref{rule:aba} and \ref{rule:aa} can be ignored since these will not arise as images of words in $\Gamma$ (they are equivalent to words in which these do not occur; recall that $xy$ and $(yx)^{-1}$ are identified). Then proceeding by induction on the applications of rule \ref{rule:cycle}, this rule corresponds precisely to inserting edges of a circuit of $\Bu$ into the word $w$.  So $w$ is in the normal closure of $\Bu$. This implies the result.
\end{proof}

\begin{corollary}
Let $\Bu$ be the collection of triangles and squares of $G$, and $T$ any spanning tree of $G$.  Then 
\[
\pi_1(G)\cong \Gamma(G,T,\Bu)
\]
where $\pi_1(G)$ denotes the fundamental group of $G$ from \cite{Grigoryan_homotopy}.
\end{corollary}


In \cite{DeVosFunkPivotto14}, DeVos, Funk, and Pivotto make use of the group we are calling $\Gamma(G,T,\Bu)$ to determine when a biased graph comes from a gain graph.  As a step in this, they prove that this group is isomorphic to the fundamental group of the topological space obtained by adding attaching a 2-cell to every circuit of $\Bu$ (see the proof of Theorem 2.1 of \cite{DeVosFunkPivotto14}).  Thus we have the following.

\begin{corollary}
Let $K$ be the 2-cell complex obtained by attaching disc to each circuit of $\Bu$.  Then $\pi_1(G,\Bu)$ is isomorphic to the fundamental group of this topological space.  In particular, the fundamental group $\pi_1(G)$ of \cite{Grigoryan_homotopy} is isomorphic to the fundamental group of the space obtained by attaching a disc to each triangle and square of $G$.
\end{corollary}
In particular, there is a canonical 1-1 correspondence between coverings of the 2-cell complex $K$, and coverings of $G$ preserving $\Bu$.
Therefore, we can now characterize the existence of an infinite connected covering of $G$ preserving $\Bu$.

\begin{corollary}
Let $G=(V,E)$ be a connected graph and let $\mathcal B$ be a set of circuits. T.f.a.e:
\begin{enumerate}
\item There exists no  infinite connected covering of $G$ preserving $\mathcal B$
\item  The fundamental group $\pi_1(G,\Bu)$ is finite.
\end{enumerate}
\end{corollary}
\begin{proof}
Due to the 1-1 correspondence between coverings of the 2-cell complex $K$ and $\Bu$ preserving coverings of $G$, the first statement is equivalent to finiteness of the universal cover of $K$, which is equivalent to finiteness of the fundamental group of $K$.
This implies the corollary since the fundamental group of $K$ is isomorphic to $\pi_1(G,\mathcal B)$.
\end{proof}

Combining this corollary with Theorem~\ref{thm:CurvCover}, we immediately obtain the following relation between curvature and the fundamental group.


\begin{corollary}\label{cor:CurvHomotopy}
Suppose a finite graph satisfies $CD(K,\infty)$ for some $K>0$. Then, $\pi_1(G)$ is finite, where $\pi_1(G)$ denotes the fundamental group of $G$ from \cite{Grigoryan_homotopy}.
\end{corollary}

We now characterize the abelianization of the fundamental group $\pi_1(G,\Bu)$.
\begin{proposition}\label{prop:ab}
Let $\Bu$ be any collection of cycles of $G$.  Then
\[\text{Ab}~\pi_1(G,\Bu) \cong \Cy(G,\Z)/\langle \Bu \rangle\]
where $\text{Ab}$ denotes the abelianization of the group, and $\langle \Bu \rangle$ denotes the set of all integer linear combinations of cycles in $\Bu$.
\end{proposition}
\begin{proof}
Let $T$ be a spanning tree.
Due to Theorem~\ref{thm:pi},
\[
\pi_1(G,\Bu) \cong  \langle \vec E \mid T,\Bu \rangle.
\]
Let $\Gamma:= \langle\vec E \mid T \rangle$.
We observe that $\text{Ab}~\Gamma \cong \Cy(G,\Z)$.
Abelianization of $\pi_1(G,\Bu)$  yields
\begin{align*}
\text{Ab}~\pi_1(G,\Bu) &\cong \langle \vec E \mid T,\Bu,\{sts^{-1}t^{-1}\} \rangle  
\cong  \frac{\langle \vec E \mid T, \{sts^{-1}t^{-1}\} \rangle }{ \langle \Bu \rangle } 
\cong \frac{\text{Ab}~\Gamma}{ \langle \Bu \rangle}
\cong \frac{\Cy(G,\Z)}{ \langle \Bu \rangle}
\end{align*}
which finishes the proof.
\end{proof}

It is known from Theorem 4.23 of \cite{Grigoryan_homotopy} that 
\[
\text{Ab}~\pi_1(G) \cong H_1(G,\Z).
\]
This result now also follows from Theorem \ref{thm:HomologyCycleSpace}, Theorem \ref{thm:pi}, and Proposition \ref{prop:ab} taken together, so we have come up with an alternative proof of this result.


\bibliographystyle{plain}
\bibliography{Bibliography}

\end{document}